\documentclass[14pt]{amsart}
\usepackage[utf8]{inputenc}
\usepackage{amsmath}
\usepackage{amsfonts}
\usepackage{amssymb}
\usepackage{physics}
\usepackage{amsthm}
\usepackage{graphicx}
\usepackage{mathrsfs}
\usepackage{enumitem}
\usepackage{float}
\usepackage{empheq}
\usepackage{setspace}
\usepackage{tikz-cd}
\usepackage{csquotes}
\usepackage{hyperref}

\renewcommand{\phi}{\varphi}

\newtheorem{definition}{Definition}[section]
\newtheorem{example}{Example}[section]
\newtheorem{theorem}{Theorem}[section]
\newtheorem{proposition}{Proposition}[section]
\newtheorem{corollary}{Corollary}[section]

\newtheorem{lemma}{Lemma}[section]

\renewcommand{\phi}{\varphi}

\title{Hyperbolicity of Multiple Ascending HNN Extensions of Free Groups}
\author{Sk Kiran Ajij}
\address{Tata Institute of Fundamental Research, Mumbai}
\email{sekh@math.tifr.res.in}

\begin{document}
  \maketitle

    \begin{abstract}
      Bestvina-Feighn-Handel show that for finitely many generic and independent hyperbolic automorphisms $\phi_1, \cdots, \phi_r$ of $F_n$, the resulting
      extension $F_n \rtimes F_r$ is hyperbolic.
      This paper generalizes the above statement to the case where  $\phi_1, \cdots, \phi_r$ are hyperbolic non-surjective endomorphisms of $F_n$.
      In our case the output is a multiple HNN extension associated to a graph with one vertex and $r$ edges. All edge
      and vertex groups are isomorphic to $F_n$.
    \end{abstract}

    \tableofcontents
    \section{Introduction}
  
Thurston~\cite{Thurs82} proved that the interior of the mapping torus $M_f$ of a hyperbolic surface $S$ with homeomorphism
$f:S\to S$ admits a finite-volume hyperbolic structure if and only if $f$ is isotopic to a
\emph{pseudo-Anosov} homeomorphism.

Motivated by Thurston’s result, Bestvina and Feighn showed that if $G$ is a word-hyperbolic group and
$\phi:G\to G$ is a \emph{hyperbolic} automorphism, then the associated mapping torus
$
G \rtimes_{\phi} \mathbb{Z}
$
is word-hyperbolic~\cite{BF92}. Subsequently, Brinkmann proved that atoroidal automorphisms of free
groups—namely, automorphisms with no nontrivial periodic conjugacy classes are hyperbolic~\cite{Br00}.


Later Patrick Reynolds~\cite{Rey} established that every irreducible, non-surjective endomorphism
of a free group admits a train track representative with no illegal turns. Based on this,
Mutanguha proved that if $\phi:F_n\to F_n$ is an irreducible, non-surjective endomorphism such that
$F_n \rtimes_{\phi} \mathbb{Z}$ contains no subgroup isomorphic to $BS(1,m)$ for any $m\ge 1$, then the
associated ascending HNN extension is word-hyperbolic~\cite{Mut20}. We refer to such non-surjective
endomorphisms as \emph{hyperbolic}.

In the paper \cite{BFH97}, Bestvina-Feighn-Handel showed that the semidirect product $F_n\rtimes F_2$ taken with respect to automorphisms $\phi_1, \phi_2: F_n\to F_n$ is hyperbolic under the hypothesis that automorphisms are independent. A similar result was shown for surface groups by Mosher in \cite{Mos97} which used 3-out-of-4 stretch lemma.

In this article, we study the hyperbolicity of HNN extensions arising from multiple hyperbolic
endomorphisms. The main idea is to show that these endomorphisms satisfy an appropriate version of
the $3$-out-of-$4$ stretch condition. A related
version of this condition was previously discussed in Section~6 of~\cite{HMS25}. The main results of the article are as follows:


  \setcounter{theorem}{0}
  \setcounter{section}{4}
  \begin{theorem}
    Let $\phi_1,\ldots,\phi_r:F_n\to F_n$ be hyperbolic,
    essentially disjoint, non-surjective endomorphisms. Then there is a $N$ such that 
    the HNN extension
    \[
      F_n \ast_{\phi_1^N,\ldots,\phi_r^N}
    \]
    is word-hyperbolic.
  \end{theorem}

  \begin{theorem}
    Let $\phi_1,\ldots,\phi_r:F_n\to F_n$ be hyperbolic,
    non-surjective, independent endomorphisms of $F_n$. Then there is a $N$ such that the HNN
    extension $F_n \ast_{\phi_1^N,\ldots,\phi_r^N}$ is word-hyperbolic.
  \end{theorem}
  \setcounter{theorem}{0}
  \setcounter{section}{1}

  \subsection*{Acknowledgements}

The author thanks his PhD supervisor, Prof.~Mahan Mj, for suggesting the problem and for helpful comments on earlier drafts. The author is also grateful to Rakesh Halder and Indranil Bhattacharyya for valuable discussions. The author thanks Mladen Bestvina for the helpful comments relayed through the supervisor. The author also thanks Jean Pierre Mutanguha for comments on the initial draft.
 The author acknowledges limited use of ChatGPT and Google Gemini for grammar and language-related corrections. The author also acknowledges the support of the Department of Atomic Energy, Government of India, under Project Identification No. RTI 4014.


  \section{Background}

  \subsection{Outer endomorphisms}

  Let $\varphi,\psi:F_n\to F_n$ be two endomorphisms of the free group
  $F_n$. They are equivalent as outer endomorphisms if there is
  $g\in F_n$ with $\phi=\iota_g\circ\psi$, where
  $\iota_g(x)=gxg^{-1}$. From now on an \emph{outer endomorphism} means an
  endomorphism up to this equivalence. Here we take $F_n$ to be a free group with
  $n\ge 2$.







    \begin{definition}[Irreducible endomorphism, \cite{Mut20_irr}]
        An (outer) endomorphism $\phi: F_n \to F_n$ is \textbf{reducible} if there exists a free factorization $A_1\ast \cdots \ast A_k \ast B $, where $B$ is nontrivial if $k=1$, and a sequence of elements, $(g_i)_{i=1}^k$, in $F_n$ such that $\phi(A_i) \le g_i A_{i+1} g_i^{-1}$ where the indices are considered $\operatorname{mod} k$. An endomorphism is
        \textbf{irreducible} if it is not reducible.
    \end{definition}
    \begin{definition}[Fully irreducible endomorphism, \cite{Mut20_irr}]
        An endomorphism $\phi: F_n\to F_n$ is \textbf{fully irreducible} if $\phi$ and all of its
        iterates are irreducible. Equivalently, $\phi$ is fully irreducible if $\phi$ has no invariant proper free factor, i.e., there does not exist a proper free factor $A\le F_n$, an element $g\in F_n$, and an integer $p\ge 1$ such that $\phi^p(A) \le g A g^{-1}$.
    \end{definition}


  \subsection{Train tracks}

  Let $\Gamma$ be a finite graph without valence-one vertices (such a graph
  is called a core graph) and with a marking $R_n\to\Gamma$, where $R_n$
  denotes a rose with $n$ edges.  The marking identifies $\pi_1(\Gamma)$
  with $F_n$, and hence any homotopy equivalence $\Gamma\to\Gamma$ determines
  an outer automorphism of $F_n$.

  A map $f:\Gamma\to\Gamma$ is a \textbf{topological representative} of
  $\phi\in\mathrm{Out}(F_n)$ if $f$ sends vertices to vertices, is
  locally injective on the interiors of edges, and the induced homomorphism
  $f_*: F_n \to F_n$ represents the outer class $\phi$.

  Given a topological representative $f:\Gamma\to\Gamma$, fix an
  ordering of edges and define the transition matrix $A(f)$ by letting
  $A(f)_{ij}$ be the number of times edge $e_i$ appears in the image of
  $e_j$ (in either direction). Such a matrix is a non-negative integral
  matrix and is \emph{irreducible} if for all $i,j$ there exists
  $N(i,j)>0$ such that the $(i,j)$-entry of $A(f)^{N(i,j)}$ is positive.

  A proper subgraph is called a \emph{forest} if all its components are
  nontrivial contractible graphs (that is, trees). We define irreducibility
  for graph maps as follows:

  \begin{definition}
    A topological representative $f:\Gamma\to\Gamma$ is irreducible if
    $\Gamma$ contains no nontrivial $f$-invariant subgraph. Equivalently,
    $f$ is irreducible if and only if its transition matrix $A(f)$ is irreducible.
  \end{definition}


  \begin{definition}[Train track]
    Let $\Gamma$ be a connected finite core graph without valence-two
    vertices. A graph map $f:\Gamma\to\Gamma$ is a \textbf{train track}
    if it maps vertices to vertices and all iterates of $f$ are locally
    injective on the interiors of edges.
  \end{definition}
  We call a train-track map
  $f:\Gamma\to\Gamma$ \emph{expanding} if it is not a homeomorphism; equivalently, the
  transition matrix is \emph{irreducible} and has a real eigenvalue $\lambda_f > 1$. The
  $\lambda_f$ is the \emph{Perron-Frobenius eigenvalue}.

  Bestvina-Handel \cite{BH92} showed that irreducible automorphisms
  admit train track representatives; the argument adapts to
  irreducible endomorphisms. See also \cite{DV96} or Theorem~2.8 in \cite{Mut20_irr}.

  \begin{theorem}
    If $\phi: F_n \to F_n$ is an irreducible endomorphism, then $\phi$ can be represented by
    an irreducible train track map. The irreducible train track is expanding iff $\phi$ has infinite-order.
  \end{theorem}

  For irreducible non-surjective endomorphisms one gets a special train
  track representative. The next result is due to Reynolds (
  \cite[Corollary~5.5]{Rey}) and can also be found in \cite{Mut24}. 
  \begin{theorem}\label{thm:inj_tt}
    Let $\phi:F_n\to F_n$ be a non-surjective irreducible endomorphism.
    Then there is a train track representative $f:\Gamma\to\Gamma$ which
    has no illegal turns.
  \end{theorem}


  \subsection{Expansive maps}

  For an automorphism $\phi:G\to G$ and any $k>0$, every $x\in G$ has
  some $y$ with $\phi^k(y)=x$. This can fail for non-surjective
  endomorphisms. For irreducible non-surjective endomorphisms a stronger
  property holds.

  \begin{definition}[Expansive, \cite{Rey}]
    Let $\mathcal{B}$ be a basis of $F_n$ and $\phi:F_n\to F_n$ an
    endomorphism. We say $\phi$ is \emph{expansive} with respect to $\mathcal B$
    if for any $K>0$ there exists $M>0$ so that for any
    $x\in F_n\setminus\{e\}$ we have
    $\|\phi^m(x)\|_{\mathcal B}\ge K$ whenever $m\ge M$.
  \end{definition}

  The expansive property is independent of the generating set chosen. There is a dichotomy for irreducible endomorphisms (
  \cite[Proposition~3.11]{Rey}).

  \begin{proposition}
    Let $\phi:F_n\to F_n$ be an irreducible endomorphism. Then either
    $\phi$ is an automorphism or $\phi$ is expansive.
    
  \end{proposition}

  \subsection{Laminations}

  A brief introduction to laminations follows; see Bestvina-Feighn-Handel
  \cite{BFH97} or \cite{Mut20_irr, lam1} for more information. Throughout this section, unless specified otherwise, we will assume that $f:\Gamma\to\Gamma$ is a train-track representative of a \emph{fully irreducible} endomorphism $\phi: F_n\to F_n$.

    Let $f: \Gamma\to\Gamma$ be a train-track map with irreducible transition matrix and let $x\in \Gamma$ be a $f$-periodic point in interior the of some edge $e$. If $\lambda > 1$ is the expansion factor of $f$ and $U$ be the $\epsilon >0$ neighborhood of $x$. Then for some $N>0$, we have that
    $f^N(U) \supset U$ and we can choose an isometric immersion $\ell: (-\epsilon, \epsilon)\to \Gamma$.
    Extend it to a locally isometric immersion $\ell: \mathbb{R}\to \Gamma$ such that $\ell(\lambda^Nt) =f^N\ell(t)$. We say $\ell$ is obtained by \emph{iterating a neighborhood} of $x$.

    \begin{definition}[\cite{BFH97}]
        Two isometric immersions $[a, b]\to \Gamma$ and $[c, d]\to \Gamma$ are called \emph{equivalent} if there is an isometry $[a, b] \to [c, d]$ making the triangle commute.
    \end{definition}
    Let $P$ represent the equivalence class of isometric immersion $\gamma: [a, b] \to \Gamma$. Then by $f(P)$ we denote the equivalence class $f\circ\gamma$ pulled tight and
    scaled so it is an isometric immersion. A \emph{leaf segment} of an isometric immersion
    $\mathbb{R}\to\Gamma$ is the equivalence class of the restriction to a finite interval.
    Suppose $\ell, \ell': \mathbb{R}\to \Gamma$ are two isometric immersions obtained from iterating the $f$-periodic points $x$ and $x'$, respectively. We say $\ell$ and $\ell'$ are \emph{(weakly) equivalent} if a leaf segment of $\ell$ is also a leaf segment of $\ell'$ and vice versa.

    The following can be found in \cite{BFH97}, as updated in the erratum \cite{BFH_err}.
    \begin{lemma}[Lemma~1.2, \cite{BFH97}]
        Let $f: \Gamma\to\Gamma$ be a train track map with a fully-irreducible transition matrix. Let $\ell$ and $\ell'$ be two isometric immersions obtained from iterating
        $f$-periodic points $x$ and $x'$ in the interior of some edge of $\Gamma$, respectively. Then $\ell$ and $\ell'$ are (weakly) equivalent.
    \end{lemma}

    \begin{definition}[Lamination, \cite{BFH97}]
        The \textbf{stable lamination} or the \textbf{attracting lamination} $\Lambda=\Lambda_{\Gamma}(f)$ is the equivalence class of isometric immersions containing some (any) immersion obtained by iterating a neighborhood of a periodic point as above. A \emph{leaf} of $\Lambda$ is an immersion representing $\Lambda$. A \emph{leaf segment} is a leaf segment of a (any) leaf $\ell$ of $\Lambda$. 
     \end{definition}

    \begin{definition}[Weak convergence, \cite{BFH97}]
        We say that a sequence $\{\alpha_i\}$ of isometric immersions of $S^1_i\to \Gamma$
        (where the metric on $S^1_i$ is the scalar multiple of standard path metric, and the scale depends on $i$) \emph{weakly converges} to $\Lambda_{\Gamma}(f)$ if the following holds ($m$ denotes the scaled Lebesgue measure):

        For every $L > 0$, the ratio
        $$
            \frac{
            m\big(\{ x \in S_i^1 \mid \text{the } L\text{-nbhd of } x \text{ is a leaf segment}         \}\big)
        }{
        m(S_i^1)
        }
        $$
        converges to $1$ as $ i \to \infty $.
    \end{definition}
    The following result can be found in \cite{BFH97}.
    \begin{proposition}
        Suppose $\alpha$ is an immersed loop in $\Gamma$ representing a $f$-non-periodic conjugacy class. Then the sequence $[f^i(\alpha)]$ of tightenings of $f^i(\alpha)$
        weakly converges to $\Lambda_{\Gamma}(f)$.
    \end{proposition}

    \begin{definition}[Quasi-periodicity, \cite{BFH97}]
        An isometric immersion $\mathbb{R}\to\Gamma$ is called \emph{quasi-periodic} if for
        every $L>0$, there exists $L'> L$ such that every leaf segment of length $L$ occurs
        as a sub-leaf segment of any leaf segment of $\ell$ of length $L'$.
    \end{definition}
    \begin{proposition}[\cite{BFH97}]
        Every leaf of $\Lambda_{\Gamma}(f)$ is quasi-periodic.
    \end{proposition}

    Now suppose $\Gamma'$ be another marked graph of same rank of $\Gamma$ and $\tau: \Gamma \to \Gamma'$ be homotopy equivalence corresponding to the difference of the two markings. For any isometric immersion $\ell: \mathbb{R} \to \Gamma$, by $\tau(\ell)$ we denote the unique (up to precomposition by an isometry of $\mathbb{R}$) isometric immersion of $\mathbb{R}\to \Gamma'$ proper homotopy equivalent to $\tau\circ\ell$. We
    have the following result from \cite{BFH97}.

    \begin{lemma}[Lemma~1.10, \cite{BFH97}]
    \begin{enumerate}
        \item If $\ell, \ell': \mathbb{R}\to \Gamma$ are equivalent, then $\tau(\ell)$ and $\tau(\ell')$ are also equivalent.
        \item If $\ell$ is quasi-periodic then so is $\tau(\ell)$.
    \end{enumerate}
    \end{lemma}
    This allows us to define the following:
    \begin{definition}[\cite{BFH97}]
        The attracting lamination of $f:\Gamma\to\Gamma$ in the $\Gamma'$-\emph{coordinates} is the equivalence class $\Lambda_{\Gamma'}(f)$ containing
        $\tau(\ell)$ for some (any) leaf $\ell$ of $\Lambda_{\Gamma}(f)$.
    \end{definition}
    \begin{lemma}[Lemma~1.11, \cite{BFH97}]
        Let $\alpha$ be a loop in $\Gamma$ represent a $f$-non-periodic conjugacy class.
        Then the sequence $\{[\tau g^i(\alpha)]\}$ weakly converges to $\Lambda_{\Gamma'}(f)$.
    \end{lemma}

    \begin{lemma}
        Let $f': \Gamma'\to\Gamma'$ be another train-track map representing the same outer
        endomorphism of $F_n$ as $f:\Gamma\to\Gamma$. Then $\Lambda_{\Gamma'}(f)=\Lambda_{\Gamma'}(f')$.
    \end{lemma}
    \begin{proof}
        The proof is identical to the automorphism case treated in \cite[Lemma~2.12]{BFH97}. 
    \end{proof}

    \begin{definition}[Attracting lamination]
        Let $\phi: F_n\to F_n$ be a fully irreducible outer endomorphism with an expanding
        train-track map $f:\Gamma\to\Gamma$. The \emph{attracting lamination} $\Lambda(\phi)$ of $\phi$ is
        the collection $\{\Lambda_{\Gamma'}(f)|\, \Gamma' \text{ is a marked \(\mathbb R\)-graph}\}$.
    \end{definition}

    \begin{definition}[Independent endomorphisms]

    Let $\phi, \psi : F_n \to F_n$ be fully irreducible outer endomorphisms.
    Let $f:\Gamma\to\Gamma$ and $g:\Gamma'\to\Gamma'$
    be expanding train–track representatives of $\phi$ and $\psi$
    on marked graphs $\Gamma$ and $\Gamma'$, respectively.
    We say that $\phi$ and $\psi$ are \emph{independent} if
    $\Lambda_\Gamma(f)$ and $\Lambda_\Gamma(g)$ are not equivalent.
    We write $\Lambda(\phi)\neq \Lambda(\psi)$ in this case.
    A collection of fully-irreducible, outer endomorphisms $\phi_1, \ldots, \phi_r: F_n\to F_n$ are called \emph{independent}, if they are mutually independent.
    \end{definition}




  \subsection{Essentially disjoint endomorphisms}

  \begin{definition}[Essentially disjoint endomorphisms]
    Let $\phi_1,\ldots,\phi_r:F_n\to F_n$ be endomorphisms. They are
    \emph{essentially disjoint} if some $N\ge 1$ satisfies
    \begin{equation*}
      \phi_i^N(F_n)\cap g\phi_j^N(F_n)g^{-1}=\{e\}
      \quad\forall g\in F_n,\ i\neq j.
    \end{equation*}
  \end{definition}

  Replacing each endomorphism by a power does not affect hyperbolicity,
  so we may assume $N=1$. An essentially disjoint collection can consist
  of non-surjective endomorphisms.

  \subsection{Pullbacks}

  Given continuous maps $f:X\to Z$ and $g:Y\to Z$ of spaces, the pullback
  $X\times_Z Y$ is defined by
  \begin{equation*}
    X\times_Z Y=\{(x,y)\in X\times Y: f(x)=g(y)\}.
  \end{equation*}

  We define pullbacks when we restrict ourselves to only graphs and graph morphism between them.
  The following definition and the statements are due to \cite{Mut20}.

  \begin{definition}[Pullback, \cite{Mut20}]
    Suppose $f:\Gamma\to\Gamma$ is a graph map. Define
    $\Gamma_i\coloneqq\Gamma\times_{f^i}\Gamma$ and let $\Gamma_0$ be
    the diagonal of $\Gamma\times\Gamma$. For $i>0$ set
    $\hat{\Gamma}_i\coloneqq\Gamma_i\setminus\Gamma_{i-1}$. We sometimes
    write $\widehat{\Gamma}(f)$ to emphasise dependence on $f$.
  \end{definition}

  There is a map $f^i:\hat{\Gamma}_j\to\hat{\Gamma}_{j-i}$ for $j>i$.
  When $f$ is injective we have:

  \begin{lemma}[cf. Lemma~3.3 in \cite{Mut20}]
    Suppose $f:\Gamma\to\Gamma$ is an immersion. Then $\Gamma_i$ is a
    union of components of $\Gamma_{i+1}$.
  \end{lemma}

  \begin{lemma}[cf. Lemma 3.4 in \cite{Mut20}]
    If $f:\Gamma\to\Gamma$ is an immersion and $\widehat{\Gamma}_i$ is
    empty, then so are $\widehat{\Gamma}_j$ for $j>i$. If
    $\widehat{\Gamma}_i$ consists of loops then so do $\widehat{\Gamma}_j$
    for $j>i$.
  \end{lemma}

  This motivates:

  \begin{definition}[\cite{Mut20}]
    For an immersion $f:\Gamma\to\Gamma$, \textbf{pullbacks stabilise} if
     $\widehat{\Gamma}_i=\varnothing$ for some $i>0$.
  \end{definition}

  \begin{proposition}[cf. Proposition 3.11 in \cite{Mut20}]
    Let $f:\Gamma\to\Gamma$ be an immersion. If $\widehat{\Gamma}_i$ is
    non-empty for all $i$ then there is an invariant loop; hence
    $\pi_1(M_f)$ contains some Baumslag--Solitar $BS(1,d)$ subgroup.
  \end{proposition}

  \begin{corollary}\label{cor:gammahat}
    Let $\phi:F_n\to F_n$ be non-surjective, irreducible, hyperbolic and
    let $f:\Gamma\to\Gamma$ be an injective train track representative.
    Then there is $N>0$ such that $\widehat{\Gamma}_1(f^N)=\varnothing$.
  \end{corollary}

  \subsection{Hyperbolic endomorphisms}

  It is a well-known result that given an atoroidal automorphism of a free group
  $\varphi: F_n\to F_n$, the extension $F_n*_{\varphi}$ is hyperbolic (\cite{Br00}).
  We have a similar result when we have endomorphisms which says:

  \begin{theorem}[\cite{Mut20}]
    Let $\phi:F_n\to F_n$ be fully irreducible, non-surjective endomorphism. The
    following are equivalent:
    \begin{enumerate}
      \item $F_n*_{\phi}$ is word-hyperbolic.
      \item $F_n*_{\phi}$ has no $BS(1,d)$ subgroups for $d\ge 1$.
      \item There are no $k,d\ge 1$ and nontrivial $g\in F_n$ with
            $[\phi^k(g)]=[g^d]$($[x]$ denoting conjugacy class of $x$).
      \item If $f:\Gamma\to\Gamma$ is a locally injective train track
            representative then $f$ is $(\lambda,n)$-hyperbolic and some
            pull-back stabilises.
    \end{enumerate}
  \end{theorem}

  \begin{definition}[Hyperbolic endomorphism]
    A fully irreducible non-surjective endomorphism
    $\phi:F_n\to F_n$ is \textbf{hyperbolic} if it satisfies any (hence
    all) of the above conditions.
  \end{definition}

  \section{Geometry of mapping torus}

  A simplified version of Bestvina-Feighn's combination theorem is
  presented here; see \cite{BF92} for more details.

  Let $\Gamma$ be a marked metric graph and let
  $f_1,\ldots,f_r:\Gamma\to\Gamma$ be graph morphisms representing
  outer-endomorphism classes $\phi_1,\ldots,\phi_r$. Form the associated mapping
  torus
  \begin{equation*}
    M(\Gamma; f_1,\ldots,f_r)\coloneqq
    \frac{\Gamma\;\coprod_{j=1}^r\Gamma_{(j)}\times[0,1]}{(x_j,1)\sim(f_j(x_j),0)}
  \end{equation*}
  where $\Gamma_{(j)}$ denotes a copy of $\Gamma$, the lower indexing is to distinguish between the
  different copies and where $x_j \in \Gamma_{(j)} \cong  \Gamma$.
  Write $M=M(\Gamma; f_1, \ldots, f_r)$ for brevity. Then $\pi_1(M)=F_n *_{j=1}^r\phi_j$.
  The space $M$ has a natural graph of space structure where the underlying graph is
  $R_r$, a rose with $r$ petals. The vertex space is
  $\Gamma$, and for each oriented edge $e_j$ of $R_r$ the corresponding
  edge map is $f_j$, while the reverse edge $\overline{e_j}$ corresponds
  to the identity map on $\Gamma$. Let
  $\pi_M:M\to R_r$ denote the natural projection. We use the notation $I_{\mathbb Z}$ to denote $I \cap \mathbb{Z}$ for any interval $I\subseteq \mathbb R$.
  \begin{definition}
    An \textbf{essential annulus} of length $m$ is a map
    \begin{equation*}
      \Delta:S^1\times[-m,m]\to M
    \end{equation*}
    satisfying:
    \begin{enumerate}
      \item It is transverse to the vertex space.
      \item The $\Delta$-preimage of the vertex space is
            $S^1\times[-m,m]_{\mathbb Z}$.
      \item For $i\in[-m,m]_{\mathbb Z}$ the loop
            $\Delta|_{S^1\times\{i\}}$ is locally injective except
            possibly at the basepoint $0\in S^1$.
      \item For $i\in[-m,m-1]_{\mathbb Z}$ the path
            $\Delta|_{\{0\}\times[i,i+1]}$ is not homotopic rel
            endpoints into the edge space.
    \end{enumerate}
  \end{definition}

  When $t\in\mathbb Z$ the based loops $\Delta_t$ are the \textbf{rings}
  and the path $\Delta^*$ is the \textbf{trace} of the basepoint, that is, $\Delta(\{0\}\times[-m ,m])$. The
  \textbf{girth} of $\Delta$ is $l(\Delta_0)$ where $l$ denotes length
  after tightening rel basepoint in the corresponding vertex space.

  For $\lambda>1$, we say $\Delta$ is $\lambda$-\textbf{hyperbolic} if
  \begin{equation*}
    \lambda l(\Delta_0)\le\max\{l(\Delta_{-m}),l(\Delta_m)\}.
  \end{equation*}

  Let $T$ be the universal cover of $R_r$. Set
  $\alpha=\pi_M\circ\Delta|_{\{0\}\times[-m,m]}:[-m,m]\to R_r$ and let
  $\alpha^*$ be a lift to $T$.

  Let $v$ be the unique vertex of $R_r$. We label the directions at $v$
  by the symbols $D_i$ and $D_i^{-1}$, where $D_i$ corresponds to the
  oriented edge $e_i$ and $D_i^{-1}$ corresponds to the reverse edge
  $\overline{e_i}$. This labeling induces a labeling of the directions at
  each vertex of the universal cover $T$ by lifting via the covering map
  $T \to R_r$.  
  
  

  For each integer $i\in[-m,m-1]$, the restriction
  $\alpha^*|_{[i,i+1]}$ traverses a single oriented edge of $T$.
  We denote by $w_i\in\{D_1^{\pm1},\ldots,D_r^{\pm1}\}$ the label of this
  edge, determined by the direction of traversal. Thus any path in $T$ canonically determines a word in the alphabet
  $\{D_1^{\pm1},\ldots,D_r^{\pm1}\}$ recording the sequence of oriented
  edges it traverses.

  In this way, the path $\alpha^*$ determines a word
  \[
    w(\Delta)=w_{-m}\,w_{-m+1}\cdots w_{m-1}
  \]
  in the free monoid generated by $\{D_1^{\pm1},\ldots,D_r^{\pm1}\}$.
  For every $i\in[-m,m-1]_{\mathbb Z}$ define $\tau_i:(i,i+1)\to\Gamma$ by
  \begin{equation*}
    \tau_i(t)=
    \begin{cases}
      x &\text{if }\Delta^*(t)=(x,s)\text{ and }s<\tfrac12,\\
      f_k(x) &\text{if }\Delta^*(t)=(x,s)\text{ and }s\ge\tfrac12\text{ and }w_i=D_k^{\pm 1}.
    \end{cases}
  \end{equation*}
  The path $\tau_i$ is the projection of
  $\alpha^*|_{(i,i+1)}$ to a cross-section of the edge space. The annulus $\Delta$ is called \emph{$\rho$-thin} if
  $l(\tau_i)+1\le\rho$ for all $i\in[-m,m-1]$ and $x\in S^1$.

  \begin{example}
    Let $\alpha$ be a locally injective loop in $\Gamma$ and let
    $\Delta:S^1\times[-1,1]\to M$ have
    $\Delta|_{S^1\times\{-1\}}=\alpha$,
    $\Delta|_{S^1\times\{0\}}=f_1(\alpha)$,
    $\Delta|_{S^1\times\{1\}}=f_2f_1(\alpha)$. Then the corresponding
    word is $D_1D_2$ (note the reverse order).
  \end{example}

  \begin{definition}[Admissible word]
    A word $w\in F\langle D_1,\ldots,D_r\rangle$ is an
    \textbf{admissible word of length $k$} if it is reduced, has length $k$
    with respect to the generating set $\{D_i,D_i^{-1}\}$, where $k$ is even,
    and of the form $w=D_{n_1}^{\pm s}D_{n_2}\cdots D_{n_{k-s-1}}$ for some
    $n_j\in\{1,\ldots,r\}$, where $s>0$.

  \end{definition}

  \begin{definition}
    Let $w$ be an admissible word of length $k$.
    \begin{enumerate}
      \item $w$ is \emph{positive} if it involves only letters $D_i$ and no
            letters $D_i^{-1}$.
      \item $w$ is \emph{unidirectional} if $w=D_j^k$ or $w=D_j^{-k}$ for some $j\in\{1, \cdots, r\}$,
            and is \emph{mixed} otherwise.
    \end{enumerate}
  \end{definition}


  Let $w$ be a positive admissible word of length $2k$, and write
  \[
    w = D_{i_{2k}} D_{i_{2k-1}} \cdots D_{i_2} D_{i_1},
  \]
  where each $i_j\in\{1,\ldots,r\}$.  
  Let $\alpha\subset\Gamma$ be a locally injective loop.

  We construct an annulus
  \[
    \Delta(\alpha,w):S^1\times[-k,k]\to M
  \]
  as follows. Set
  \[
    \Delta_{-k}=\alpha,
  \]
  and define inductively, for $j=1,\ldots,2k$,
  \[
    \Delta_{-k+j} = f_{i_j}(\Delta_{-k+j-1}),
  \]
  where each $\Delta_t$ is viewed as a loop in the corresponding vertex
  space of $M$.

  The trace of the basepoint $\Delta(\{0\}\times[-k,k])$ projects under
  $\pi_M:M\to R_r$ to a path $\alpha:[-k,k]\to R_r$ whose lift
  $\alpha^*:[-k,k]\to T$ traverses, on each interval $[j-1,j]$, the
  oriented edge of $T$ labelled by $D_{i_j}$.

  This construction produces a uniformly thin annulus, in particular,
  $\Delta(\alpha,w)$ is $1$-thin. We now do the similar annuli construction for
  bi-directional word in the following way. Suppose, without loss of generality that
  \[
    w = D_1^{-s} D_{i_1} \cdots D_{i_k}
  \]
  is an admissible word with $s\ge 1$ and of total word length $2m$. Assume that $\alpha\subset\Gamma$
  is a locally injective loop for which there exists a loop
  $\beta\subset\Gamma$ satisfying
  \[
    f_1^{\,s}(\beta) = \alpha
  \]
  (up to free homotopy).
  
  We define an annulus $\Delta=\Delta(\alpha,w):S^1\times[-m,m]\to M$ by
  setting
  \[
    \Delta_{-m} = \alpha = f_1^{\,s}(\beta),
  \]
  and inductively defining the remaining rings $\Delta_{-m+1}, \cdots \Delta_{-m+s}$
  \[
    f_1^{\,s-1}(\beta), \cdots, f_1(\beta), \beta
  \]
  while the remaining letters $D_2\cdots D_k$ are constructed by successive
  applications of the corresponding maps $f_{i_1},\ldots,f_{i_k}$. In this way,
  the trace of the basepoint projects to a path whose associated word is
  $w$, and $\Delta(\alpha, w)$ is an annulus with associated word $w$.


  \begin{proposition}\label{prop:admis}
      Let $\phi_1, \cdots, \phi_r: F_n \to F_n$ be a family of essentially disjoint endomorphisms. Then, there exists a $N\ge 1$ such that after replacing each $\phi_i$
      by $\phi^N_i$ and constructing the mapping torus $M$ as above, the word associated to
      any $\rho$-thin annulus $\Delta: S^1\times [-m, m] \to M$ is admissible.
  \end{proposition}

  \begin{proof}
    Since the family of endomorphisms $\phi_1, \cdots \phi_r: F_n\to F_n$ is essentially disjoint, we take a $N\ge 1$ large enough so that
    \begin{equation*}
      \phi_i^N(F_n)\cap g\phi_j^N(F_n)g^{-1}=\{e\}
      \quad\forall g\in F_n,\ i\neq j.
    \end{equation*}
    and the conclusion of Corollary~\ref{cor:gammahat} is satisfied for all $i$.
    After replacing each $\phi_i$ by $\phi_i^N$, we may assume $N=1$. Now we show that no subword of the form $D_iD_j^{-1}$ with $i\neq j$
    can occur in $w(\Delta)$. Suppose such a subword appears. Restricting
    $\Delta$ to the corresponding subannulus, we obtain an annulus
    $\Delta:S^1\times[-1,1]\to M$ with rings
    $\Delta_{-1},\Delta_0,\Delta_1$ and, without loss of generality, let the
    associated word be $D_1D_2^{-1}$.

    By construction of the mapping torus, the loops
    $f_1(\Delta_{-1})$ and $f_2(\Delta_1)$ are homotopic to $\Delta_0$
    in the vertex space. Hence they represent the same conjugacy class in
    $F_n$. Let $x,y\in F_n$ represent the conjugacy classes of
    $\Delta_{-1}$ and $\Delta_1$, respectively. Then
    $\phi_1(x)$ and $\phi_2(y)$ are conjugate to the same nontrivial element
    of $F_n$, contradicting the essential disjointness of the endomorphisms.
    Therefore no subword $D_iD_j^{-1}$ with $i\neq j$ can occur in
    $w(\Delta)$.
    
    Next we show that $w(\Delta)$ is reduced, i.e.\ it contains no subwords
    of the form $D_iD_i^{-1}$ or $D_i^{-1}D_i$. Suppose, for contradiction,
    that $w(\Delta)$ contains the subword $D_1D_1^{-1}$. Restricting to the
    corresponding subannulus, we may assume
    that $\Delta:S^1\times[-1,1]\to M_{f_1}$ is $\rho$-thin and without loss
    of generality, the associated word is $D_1D_1^{-1}$.

    By homotoping $\Delta$ rel.\ base point within the edge space, we may
    further assume that $\Delta$ is $1$-thin with associated word $D_1D_1^{-1}$. Thus the pair
    $(\Delta_{-1},\Delta_1)$ determines a nontrivial loop in the first
    pullback $\widehat{\Gamma}_1(f_1)$. This contradicts the assumption that
    $\widehat{\Gamma}_1(f_1)=\varnothing$. Therefore no such subword can
    occur, and $w(\Delta)$ is reduced.

    Combining the two steps, we conclude that $w(\Delta)$ is reduced and
    contains no subwords of the form $D_iD_j^{-1}$ with $i\neq j$.
    Therefore $w(\Delta)$ is admissible.

  \end{proof}


  \begin{lemma}\label{lem:kK}
    Let $\Gamma,\Gamma'\in CV_n$ and let $h:\Gamma\to\Gamma'$ be a
    difference of markings with homotopy inverse $h'$. Let
    $\alpha\subset\Gamma$ be a locally injective loop, or a based loop
    locally injective except possibly at the basepoint. Then there exists
    $K>0$ such that
    \begin{equation*}
      K^{-1}\,l(h(\alpha)) \le l(\alpha) \le K\,l(h(\alpha)),
    \end{equation*}
    where $l(\beta)$ denotes the length of $\beta$.
  \end{lemma}

  \begin{proof}
    Let $\sigma(h)\coloneqq\max_{e\in E(\Gamma)}\frac{|h(e)|}{|e|}$ and
    likewise $\sigma(h')$. Since loops are concatenations of edges,
    $l(h(\alpha))\le\sigma(h)l(\alpha)$. Because $\alpha$ is locally
    injective and $h'h(\alpha)$ is homotopic to $\alpha$ we have
    $l(h'h(\alpha))\ge l(\alpha)$. Hence
    \begin{align*}
      l(\alpha) &\le l(h'h(\alpha)) \le \sigma(h')l(h(\alpha)).
    \end{align*}
    Setting $K=\max\{\sigma(h),\sigma(h')\}$ and combining the inequalities
    yields the claim. The same argument applies to based loops.
  \end{proof}

  \section{Annuli Flaring and Hyperbolicity}

  An automorphism $\phi:G\to G$ is \textbf{hyperbolic} if there exist
  $\lambda>1$ and $n\ge 1$ such that for all nontrivial $g\in G$,
  \begin{equation*}
    \lambda\,l(g)\le \max\{l(\phi^n(g)),l(\phi^{-n}(g))\}.
  \end{equation*}
  where $l(\cdot)$ is word length in a fixed finite generating set. For
  endomorphisms inverses need not exist.

  A graph map $f:\Gamma\to\Gamma$ is assumed to send vertices to vertices and is
  locally injective on interiors of edges. It maps edges to nontrivial
  edge-paths but may not be locally injective at vertices. Below we define
  $(\lambda, n)$-hyperbolic graph maps which extend the definition given in
  \cite{Mut20}.



  \begin{definition}[$(\lambda,n)$-hyperbolic]
    For $\lambda>1$, $n\ge 1$ and graph maps
    $f_1,\ldots,f_r:\Gamma\to\Gamma$, the collection of graph maps $\{f_1, \cdots, f_r\}$ is
    $(\lambda,n)$-hyperbolic if for any admissible word $w$ of length
    $2n$ and any locally injective loop (or based loop)
    $\alpha\subset\Gamma$, the associated annulus
    $\Delta=\Delta(\alpha,w):S^1\times[-n,n]\to M$ satisfies one of:
    \begin{enumerate}
      \item $\lambda l_0\le l_n$,
      \item $\lambda l_0\le l_{-n}$,
    \end{enumerate}
    where $l_j=l(\Delta|_{S^1\times\{j\}})$.
  \end{definition}

  Now fix notation. For an irreducible non-surjective endomorphism
  $\phi_1$ there is a train track representative $f_1:\Gamma\to\Gamma$
  with no illegal turns by Theorem \ref{thm:inj_tt}. For
  $\phi_2,\ldots,\phi_r$ choose train track representatives
  $f_j':\Gamma_j\to\Gamma_j$. Find graph maps $h_j:\Gamma\to\Gamma_j$ and
  $f_j:\Gamma\to\Gamma$, obtained by suitable change of markings, with $h_j\circ f_j=f_j'\circ h_j$.
  \[
    \begin{tikzcd}
      \Gamma \arrow{r}{f_j} \arrow{d}[swap]{h_j}
      & \Gamma \arrow{d}{h_j} \\
      \Gamma_j \arrow{r}{f_j'}
      & \Gamma_j
    \end{tikzcd}
  \]

    The following lemma is adapted from Lemma~6.3 of \cite{Mut20}.
    \begin{lemma}[Uniform expansion with coefficient $3$]\label{lem:31hyp}
      Let $f:\Gamma \to \Gamma$ be a train track representative of a non-surjective hyperbolic endomorphism of $F_n$. Then there exists $N \ge 1$ such that for every immersed loop $\alpha \subset \Gamma$,
      \[
      \ell\!\left(f^{N}(\alpha)\right) > 3\,\ell(\alpha).
      \]
      \end{lemma}
  \begin{proof}
    Since $f$ represents a hyperbolic endomorphism, it has no periodic conjugacy classes and hence no invariant immersed loop.

    First suppose that $\Gamma$ contains no nontrivial $f$-invariant forest. If there exists an edge $e$ such that $l(f^n(e)) = 1$ for all $n \ge 1$, then the union
    \[
    \bigcup_{n \ge 1} f^n(e)
    \]
    is a finite $f$-invariant subgraph of $\Gamma$. Since this subgraph is not a forest, it contains an immersed loop whose iterates under $f$ have constant length and since there are only finitely many edges, it is eventually periodic
    under $f$, contradicting hyperbolicity. Hence, for every edge $e$ there exists $n_e \ge 1$ such that $\ell(f^{n_e}(e)) \ge 3$. Let
    \[
    N = \max\{n_e : e \text{ is an edge of } \Gamma\}.
    \]
    Then $\ell(f^{N}(e)) \ge 3$ for every edge $e$. Since $f$ is a train track map with no illegal turns, we get that
    \[
    l(f^N(\alpha)) \ge 3\,l(\alpha).
    \]

    Now suppose that $\Gamma$ contains a nontrivial $f$-invariant forest $F$. Collapse each component of $F$ to obtain a graph $\Gamma'$ and an induced map $f' : \Gamma' \to \Gamma'$. Since $f$ is immersion, $f$ restricted to the component of the maximal forest is a homeomorphism. Suppose $v'\in\Gamma'$ is the image of
    a vertex of $\Gamma$ whose $f$-image is in a collapsed tree. The directions at $v'$, $T_{v'}\Gamma'$, correspond to directions at the
    boundary of the collapsed tree $T_{\partial Y}\Gamma'$ if $v'$ itself is the image of the collapsed tree $Y$, otherwise they
    correspond to the directions of some vertex of $\Gamma$ outside collapsed forest. Since $f$ is injective on the
    boundary of the collapsed tree, the union of $df$ is injective and hence $df'_{v'}$ is also injective. Clearly $df'_{v'}$ is
    injective if $v'$ is the image of a vertex of $\Gamma$ whose $f$-image lies outside the collapsed forest. We hence get that $f'$ is
    also injective. By the previous case, there exists $N' \ge 1$ such that for every immersed loop $\alpha' \subset \Gamma'$,
    \[
    l(f'^{\,N'}(\alpha')) \ge 3\,l(\alpha').
    \]
    
    Choose $k \ge 1$ so that $3^k$ exceeds the number of edges in the maximal invariant forest $F$, and set $N = kN'$. Then any immersed loop $\alpha \subset \Gamma$ contains an edge outside $F$, whose image under $f^N$ has length exceeding $3l(\alpha)$. Hence
    \[
    l(f^N(\alpha)) \ge 3\,l(\alpha),
    \]
    as required.
  \end{proof}

  \begin{proposition}\label{prop:hypann}
    With the above notation, if the endomorphisms are hyperbolic then,
    after possibly taking suitable powers, the mapping torus
    $M=M(\Gamma;f_1,\ldots,f_r)$ is $(3,1)$-hyperbolic.
  \end{proposition}

  \begin{proof}
    Each $f_j'$ is a train-track representative and $\phi_j$ satisfies
    that there is a $N_j$ such that $l(f'^{N_j}_j(\alpha)) \ge 3K_j^2l(\alpha)$ for any immersed
    loop or based loop $\alpha \subset \Gamma_j$ as in Lemma \ref{lem:31hyp}. By replacing each $f_j'$ by
    suitable power, we can assume that $N_j=1$. Let $K_j$ be as in Lemma
    \ref{lem:kK}. Then
    \begin{gather*}
      l(h_jf_j(\alpha)) = l(f_j'h_j(\alpha))
      \ge 3K_j^2l(h_j(\alpha)) \ge 3K_jl(\alpha),\\
      \implies K_jl(f_j(\alpha)) \ge 3K_jl(\alpha),\\
      \implies l(f_j(\alpha))\ge 3l(\alpha).
    \end{gather*}
    We now verify the $(3,1)$-hyperbolicity condition.
    Let $w$ be an admissible word of length $2$ and let
    $\Delta=\Delta(\alpha,w):S^1\times[-1,1]\to M$ be the associated annulus.
    By admissibility, $w$ is either of the form $D_jD_k$, $D_j^{-1}D_k$ or $D_j^{-2}$
    with no cancellation. Reversing the orientation of the
    annulus replaces $w$ by its inverse. Since the $(3,1)$-hyperbolicity condition
    is symmetric in $l_{-1}$ and $l_1$, we may assume without loss of
    generality that $w=D_jD_k$ or $w=D_j^{-1}D_k$.

    If $w=D_jD_k$, then $\Delta_1$ is freely homotopic to
    $f_k(f_j(\alpha))$, and hence
    \[
      l_1 = l(\Delta_1) \ge l(f_k(f_j(\alpha)))
      \ge 3\,l(f_j(\alpha))=3l_0.
    \]
    If $w=D_j^{-1}D_k$, then similarly $f_j(\Delta_{0})$ is freely homotopic to
    $\Delta_{-1}$, and therefore
    \[
      l_{-1} \ge l(f_j(\Delta_0)) \ge 3l(\Delta_0)=3l_0
    \]
    In either case, the annulus $\Delta$ satisfies the $(3,1)$-hyperbolicity
    inequality.
  \end{proof}

  \begin{theorem}\label{th:bfhyp}
    Let $f_1,\ldots,f_n:\Gamma\to\Gamma$ be graph morphisms as before.
    Then $M$ satisfies the annuli flaring condition.
  \end{theorem}

  The statement also follows from \cite[Theorem~6.4]{Mut20} after a slight modification. We
  give another proof.

  \begin{proof}
    By Proposition \ref{prop:admis}, for any $\rho$-thin annulus $\Delta$ of
    length $1$ the associated word $w=w(\Delta)$ is admissible. Without loss of
    generality and up to reversing the orientation of the annuli, such a word has form $D_j^{\pm1}D_1$ and set
    $f=f_1$. We have $\Delta_1\simeq\tau_0 f(\Delta_0)\overline{\tau_0}$
    and therefore
    \begin{gather*}
      l(\Delta_1)\ge l(f(\Delta_0))-2\rho.
    \end{gather*}
    Now since $f=f_1$ satisfies $l(f(\alpha)) \ge 3l(\alpha)$, we get
    \begin{equation*}
      l_1\ge 3l_0-2\rho.
    \end{equation*}
    If the girth $l_0>2\rho$ then $l_1\ge 2l_0$ and the Bestvina-Feighn
    flaring condition holds with $\lambda=2$.
  \end{proof}

  \section{Proof of main theorem}


  \begin{theorem}\label{thm:main1}
    Let $\phi_1,\ldots,\phi_r:F_n\to F_n$ be hyperbolic,
    essentially disjoint, non-surjective endomorphisms. Then there is a $N$ such that 
    the HNN extension
    \[
      F_n \ast_{\phi_1^N,\ldots,\phi_r^N}
    \]
    is word-hyperbolic.
  \end{theorem}



  \begin{proof}
    Following the notation of Section~3, choose $N$ sufficiently large so
    that all associated words of $\rho$-thin annuli are admissible, as
    ensured by Proposition~\ref{prop:admis}. Let
    \[
      M=M(\Gamma;f_1,\ldots,f_r)
    \]
    be the mapping torus, where each $f_i$ is a representative of
    $\phi_i^N$. By Proposition~\ref{prop:hypann}, the mapping torus $M$
    satisfies $(3,1)$-hyperbolicity. By Theorem~\ref{th:bfhyp}, $M$
    satisfies the annuli flaring condition, and hence
    \[
    \pi_1(M)=F_n \ast_{\phi_1^N,\ldots,\phi_r^N}
    \]
    is word-hyperbolic.
  \end{proof}

    The following results will be used to prove the next theorem.

    \begin{lemma}\label{lem:weak_conv}
    Let $\phi:F_n\to F_n$ be a non-surjective, fully irreducible
    endomorphism and let $f:\Gamma\to\Gamma$ be its expanding
    train–track representative, which has no illegal turns (according to \ref{thm:inj_tt}).
    Let $w_k\in \phi^k(F_n)$ and let $\alpha_k$ be immersed loops
    representing the conjugacy classes of $w_k$ in $\Gamma$.
    Then $\{\alpha_k\}$ weakly converges to the attracting
    lamination $\Lambda_\Gamma(f)$.
    \end{lemma}
    
\begin{proof}
For each $k$ choose $u_k\in F_n$ such that
$w_k=\phi^k(u_k)$ and let $\beta_k$ be an immersed loop
representing $u_k$ in $\Gamma$.
Then, after tightening,
\[
\alpha_k = f^k(\beta_k).
\]

Since $f$ is an expanding train–track map with primitive
transition matrix, for each edge $e$ of $\Gamma$ we have
$l(f^k(e))\to\infty$ as $k\to\infty$.
Define
\[
\delta_k = \min\{l(f^k(e)) \mid e \text{ an edge of } \Gamma\}.
\]
Then $\delta_k\to\infty$.

If $\beta_k$ crosses $m_k$ edges, and since all turns are legal, we have
\[
l(\alpha_k)
=
l(f^k(\beta_k))
=
\sum_{i=1}^{m_k} l(f^k(e_i))
\ge m_k \delta_k.
\]

Fix $L>0$.
Because $f$ has no illegal turns, every point of $\alpha_k$
whose $L$–neighbourhood does not form a leaf segment
must lie within distance $L$ of the image of a vertex
of $\beta_k$.
There are $m_k$ such vertices, and each contributes at most
$2L$ to the total length of bad points.
Hence the measure of the set of points whose $L$–neighbourhood
is not a leaf segment is at most $2L m_k$.

Therefore
\[
\frac{
m\big(\{x\in \alpha_k \mid L\text{–nbhd of }x
\text{ is not a leaf segment}\}\big)
}{
l(\alpha_k)
}
\le
\frac{2L m_k}{m_k \delta_k}
=
\frac{2L}{\delta_k}.
\]

Since $\delta_k\to\infty$, the right-hand side tends to $0$.
Thus for every $L>0$,
\[
\frac{
m\big(\{x\in \alpha_k \mid L\text{–nbhd of }x
\text{ is a leaf segment}\}\big)
}{
l(\alpha_k)
}
\to 1.
\]
Therefore the immersed loops $\alpha_k$ converge weakly to the lamination $\Lambda_{\Gamma}(f)$ as claimed.
\end{proof}

\begin{proposition}\label{prop:indep_disj}
Suppose $\phi, \psi :F_n\to F_n$ are fully irreducible,
non-surjective, hyperbolic, independent endomorphisms of $F_n$.
Then $\phi$ and $\psi$ are essentially disjoint.
\end{proposition}

\begin{proof}
Suppose, for a contradiction, that $\phi$ and $\psi$
are not essentially disjoint.
Then for every integer $k\ge 1$, there exists
$g_k\in F_n$ such that
\[
\phi^k(F_n)\cap g_k\,\psi^k(F_n)\,g_k^{-1}
\]
contains a nontrivial element $w_k$.

Choose nontrivial $u_k,v_k\in F_n$ such that
\[
w_k=\phi^k(u_k)=g_k\,\psi^k(v_k)\,g_k^{-1}.
\]

Let $f:\Gamma\to\Gamma$ and $g:\Gamma'\to\Gamma'$
be expanding train–track representatives of
$\phi$ and $\psi$, respectively.
Let $\alpha_k$ be a cyclically reduced immersed loop
in $\Gamma$ representing the conjugacy class $[w_k]$,
and let $\alpha'_k$ be a cyclically reduced immersed loop
in $\Gamma'$ representing the same conjugacy class.

By Lemma~\ref{lem:weak_conv}, the sequence $\{\alpha_k\}$
weakly converges to $\Lambda_\Gamma(f)$, and the sequence
$\{\alpha'_k\}$ weakly converges to $\Lambda_{\Gamma'}(g)$.

Let $\tau:\Gamma'\to\Gamma$ be change of markings.
Since $\tau$ preserves immersed loops up to tightening,
we have
\[
[\tau(\alpha'_k)] = [\alpha_k].
\]
Weak convergence is preserved under change of marking,
and therefore $\{\alpha_k\}$ also weakly converges
to $\Lambda_\Gamma(g)$.

Hence the sequence $\{\alpha_k\}$ weakly converges
simultaneously to $\Lambda_\Gamma(f)$ and to
$\Lambda_\Gamma(g)$.

Let $L>0$ and let $\gamma$ be a leaf segment of
$\Lambda_\Gamma(f)$ of length $L$.
By quasi-periodicity of $\Lambda_\Gamma(f)$,
there exists $L'>L$ such that every leaf segment
of length $L'$ contains $\gamma$ as a subsegment.

Since $\{\alpha_k\}$ weakly converges to both laminations,
for sufficiently large $k$ the loop $\alpha_k$
contains a segment $\gamma'$
of length $L'$ which is a leaf segment of
both $\Lambda_\Gamma(f)$ and $\Lambda_\Gamma(g)$.
As $\gamma'$ contains $\gamma$, it follows that
$\gamma$ is a leaf segment of $\Lambda_\Gamma(g)$.

Since $L>0$ was arbitrary,
every leaf segment of $\Lambda_\Gamma(f)$
occurs in $\Lambda_\Gamma(g)$.
By symmetry, the reverse inclusion holds.
Hence
\[
\Lambda_\Gamma(f)=\Lambda_\Gamma(g),
\]
contradicting independence.

Therefore $\phi$ and $\psi$ are essentially disjoint.
\end{proof}

  \begin{theorem}\label{thm:main2}
    Let $\phi_1,\ldots,\phi_r:F_n\to F_n$ be hyperbolic,
    non-surjective, independent endomorphisms of $F_n$. Then there is a $N$ such that the HNN
    extension $F_n*_{j=1}^r{}_{\phi_j^N}$ is word-hyperbolic.
  \end{theorem}



  \begin{proof}
    
    Since the endomorphisms $\phi_1,\ldots,\phi_r$ are independent,
    Proposition~\ref{prop:indep_disj} implies that for each pair
    $i\neq j$ there exists $N_{ij}\ge 1$ such that
    \[
        \phi_i^{N_{ij}}(F_n)\cap g\,\phi_j^{N_{ij}}(F_n)\,g^{-1}
        =\{e\}
        \quad\text{for all } g\in F_n.
    \]
    Let $N=\max\{N_{ij}\mid i\neq j\}$.
    Then for this $N$, the endomorphisms
    $\phi_1,\ldots,\phi_r$ are \emph{essentially disjoint}.
    Thus the hypotheses of Theorem~\ref{thm:main1}
    are satisfied, and the HNN extension
    \[
    F_n*_{j=1}^r{}_{\phi_j^N}
    \]
    is word-hyperbolic for some large $N$.
  \end{proof}


  \bibliography{references}
  \bibliographystyle{plain}








\end{document}